\documentclass[12pt]{amsart}\usepackage{amsfonts,latexsym}

\newtheorem{theorem}{Theorem}[section]
\newtheorem{lemma}[theorem]{Lemma}
\newtheorem{thm}[theorem]{Theorem}

\theoremstyle{definition}
\newtheorem{definition}[theorem]{Definition}
\newtheorem{definitions}[theorem]{Definitions}

\newtheorem{problem}[theorem]{Problem}

\theoremstyle{remark}
\newtheorem{remark}[theorem]{Remark}

\DeclareMathOperator{\alm}{alm}

\title[Combinatorial principles, compactness of spaces V]
{Combinatorial and model-theoretical principles related to 
regularity of ultrafilters and compactness of topological spaces. V.}

\author[]{Paolo Lipparini} 
\address{Dipartimento di Matematica\\
Viale della Ricerca Scientifica\\
II Universitoma di Roma (Tor Vergata)\\
I-00133 ROME 
ITALY
}
\urladdr{http://www.mat.uniroma2.it/\textasciitilde lipparin}
\thanks{The author has received support from MPI and GNSAGA.
We wish to express our gratitude to X. Caicedo for stimulating discussions and correspondence} 


\begin{document} 

\begin{abstract} 
We generalize to the relations
$(\lambda, \mu) \stackrel{\kappa}{\Rightarrow} (\lambda', \mu')$ and
$\alm (\lambda, \mu) \stackrel{\kappa}{\Rightarrow} \alm (\lambda', \mu')$
 some results obtained in 
Parts II and IV. We also present a multi-cardinal version.
\end{abstract}

\maketitle 

In this note we present a version of 
\cite[II, Theorem 1]{parti} and
\cite[IV, Theorem 7]{parti}
for the relation $(\lambda, \mu) \stackrel{\kappa}{\Rightarrow} (\lambda', \mu')$ and some variants.

See Parts I, II, IV \cite{parti} or 
\cite{BF,CN,CK,KM,easter,bumi,arxiv}
for unexplained notation.

\section{Equivalents of $(\lambda, \mu) \stackrel{\kappa}{\Rightarrow} (\lambda', \mu')$}\label{1} 

Let us recall the definition of $(\lambda, \mu) \stackrel{\kappa}{\Rightarrow} (\lambda', \mu')$, as given in \cite{abst}. 
Another formulation (equivalent for $\kappa \geq \sup \{  \lambda , \lambda' \}$) had been given in 
\cite{easter}. Notice that the two formulations 
are \textbf{not necessarily equivalent} for $ \kappa < \sup \{  \lambda , \lambda' \}$. See also \cite[Section 0]{arxiv}, and Definition \ref{deffromarx} and
Theorem \ref{lmklmeq} below.

Given infinite cardinals $ \lambda  \geq \mu$ and a set $y\in S_\mu(\lambda )$,
we denote by $[y] _{S_\mu(\lambda )} $, or simply by
$ [y] $ if there is no danger of confusion, the 
``cone'' $\{ s \in S_\mu(\lambda )| y \subseteq s\}$
 of $ y $ in $S_\mu(\lambda )$. 

\begin{definition}\label{lmklm} \textrm{\cite{abst}} 
We say that an ultrafilter $D$ over $S_\mu(\lambda )$
\emph{covers} $\lambda $ if and only if 
$[\{ \alpha \}] = \{ s \in S_\mu(\lambda )| \alpha\in s\} \in D $,
for every $ \alpha \in \lambda $.

Let  $ \lambda \geq \mu$, $ \lambda' \geq \mu'$
 be infinite cardinals, and $ \kappa $ be any cardinal. 

The relation $(\lambda, \mu) \stackrel{\kappa}{\Rightarrow} (\lambda', \mu')$ 
holds if and only if 
there are $\kappa $ functions 
$(f_\beta )_{\beta \in \kappa }: S_\mu(\lambda ) \to S_{\mu'}(\lambda ')$
such that whenever $D$ is an ultrafilter
over $S_\mu(\lambda )$, and $D$ covers $\lambda $,
then for some $\beta\in\kappa $ it happens that
$f_\beta (D)$ covers $ \lambda '$.
    \end{definition}

In order to state the next theorem, we must recall some definitions from \cite{easter} and \cite[Section 0]{arxiv}.

\begin{definition}\label{deffromarx}
${\mathfrak S}(\lambda, \mu; \lambda',\mu') $ denotes the model
$\langle S_{\mu''}(\lambda ''), \subseteq, U, U', \{\alpha \} \rangle_{\alpha\in \lambda''}$,
where
$\lambda''= \sup  \{ \lambda, \lambda' \} $,
$\mu''= \sup  \{ \mu, \mu' \} $,
$U(x)$ if and only if $x \in S _\lambda (\mu)$, and
$U'(y)$ if and only if $y\in  S _{ \lambda'} (\mu')$. 

If ${\mathfrak S}^+ $ is an expansion of 
${\mathfrak S}(\lambda, \mu; \lambda'\mu') $,
and $ {\mathfrak B} \equiv {\mathfrak S}^+ $,
we say that ${\mathfrak B}$ is $(\lambda, \mu)$-\emph{regular}  
if and only if there is $b \in B$ such that 
$ {\mathfrak B} \models U(b) $, and
$ {\mathfrak B} \models  \{ \alpha \} \subseteq   b $
for every $ \alpha \in \lambda $.

Similarly, ${\mathfrak B}$ is $(\lambda', \mu')$-\emph{regular}  
if and only if there is $b' \in B$ such that 
$ {\mathfrak B} \models U'(b') $, and
$ {\mathfrak B} \models  \{ \alpha'  \} \subseteq b' $
for every $ \alpha' \in \lambda' $.
\end{definition} 

\begin{thm}\label{lmklmeq} 
Suppose that  $ \lambda \geq \mu$, $ \lambda' \geq \mu'$
 are infinite cardinals, and $ \kappa $ is any cardinal. 

Then the following conditions are equivalent.

\smallskip

(a) $(\lambda, \mu) \stackrel{\kappa}{\Rightarrow} (\lambda', \mu')$ holds.

\smallskip

(b) There are $ \kappa $ functions 
$ (f_ \beta ) _{ \beta \in \kappa }: S_\mu(\lambda ) \to S_{\mu'}(\lambda ')$
such that 
for every function $g: \kappa \to \lambda' $ there exist finite
sets $F \subseteq \kappa $ and   $G \subseteq \lambda $ such that 
$ [G] \subseteq
 \bigcup_{\beta\in F}  f_\beta ^{-1} [\{g(\beta)\}]$.

\smallskip

(c) There is a family $ (C_{ \alpha  , \beta }) _{ \alpha\in\lambda'  , \beta \in \kappa}  $ 
of subsets of $ S_ \mu(\lambda)$ such that:

(i) For every $ \beta\in\kappa$ and  every  $H \subseteq \lambda' $, if
 $|H|\geq \mu'$ then
 $\bigcap_{\alpha  \in H} C_{ \alpha  , \beta  } = \emptyset  $.

 (ii) For every function $g: \kappa \to \lambda' $ there exist finite
sets $F \subseteq \kappa $ and   $G \subseteq \lambda $ such that 
$ [G] \subseteq  \bigcup_{\beta\in F} C_{ g(\beta) , \beta  }$.

 \smallskip

(c$'$) There is a family $ (B_{ \alpha  , \beta }) _{ \alpha\in\lambda'  , \beta \in \kappa}  $ 
of subsets of $ S_ \mu(\lambda)$ such that:

(i) For every $ \beta\in\kappa$ and  every  $H \subseteq \lambda' $, if
 $|H|\geq \mu'$ then
 $\bigcup_{\alpha  \in H} B_{ \alpha  , \beta  } = S_ \mu(\lambda) $.

 (ii) For every function $g: \kappa \to \lambda' $ there exist finite
sets $F \subseteq \kappa $ and   $G \subseteq \lambda $ such that 
$ [G] \cap \bigcap_{\beta\in F} B_{ g(\beta) , \beta  } 
= \emptyset $.

\smallskip

If in addition
$\kappa \geq \sup \{  \lambda , \lambda' \} $,
then the preceding conditions are also equivalent to:

\smallskip

(d) ${\mathfrak S}(\lambda, \mu; \lambda',\mu') $ has 
a multi-sorted expansion
${\mathfrak S}^+ $ 
with at most $\kappa $ new symbols such that whenever 
$ {\mathfrak B} \equiv {\mathfrak S}^+ $ and 
${\mathfrak B}$ is $(\lambda, \mu)$-regular, then 
${\mathfrak B}$ is $(\lambda', \mu')$-regular.

\smallskip

(e) ${\mathfrak S}(\lambda, \mu; \lambda',\mu') $ has 
an expansion
${\mathfrak S}^+ $ 
with at most $\kappa $ new symbols such that whenever 
$ {\mathfrak B} \equiv {\mathfrak S}^+ $ and 
${\mathfrak B}$ is $(\lambda, \mu)$-regular, then 
${\mathfrak B}$ is $(\lambda', \mu')$-regular.

\smallskip

If in addition
$\kappa \geq \sup \{  \lambda , \lambda' \} $, and $\lambda'=\mu'$ is a regular cardinal,
then the preceding conditions are also equivalent to:

\smallskip

(f) Every $\kappa$-$(\lambda,\mu)$-compact logic is  
$\kappa$-$(\lambda',\lambda')$-compact. 

\smallskip

(g) Every $\kappa$-$(\lambda,\mu)$-compact logic generated by $\lambda' $
cardinality quantifiers is $(\lambda',\lambda')$-compact. 
\end{thm} 
 
\begin{remark}\label{announced}
The equivalence of conditions 
(a) and (e) above, for $\kappa \geq \sup \{  \lambda , \lambda' \} $,
 has been announced in 
\cite{abst}. 

For $\kappa=\lambda=\mu$ and  $\lambda'=\mu'$ both regular cardinals, 
the equivalence of Conditions
(e), (f), (g) above has been announced in
\cite[p. 80]{arxiv}. 
\end{remark}

\begin{lemma}\label{lem}
Suppose that  $ \lambda \geq \mu$ and $ \lambda' \geq \mu'$ 
are infinite cardinals, and 
$\kappa $ is any cardinal.
Suppose that $ (f_ \beta ) _{ \beta \in \kappa } $
is a given set of functions from
$S_\mu(\lambda )$ to $S_{\mu'}(\lambda ')$.

Then the following are equivalent.

(a) Whenever $D$ is an ultrafilter
over $S_\mu(\lambda )$ and $D$  covers 
 $\lambda $, then there exists some $ \beta \in \kappa $
such that $f_ \beta (D)$ covers $ \lambda' $.

(b) For every function $g: \kappa \to \lambda' $ there exist finite
sets $F \subseteq \kappa $ and   $G \subseteq \lambda $ such that 
$ [G] \subseteq
 \bigcup_{\beta\in F}  f_\beta ^{-1} [\{g(\beta)\}]$.  
\end{lemma}

\begin{proof}
We show that the negation of (a) is equivalent to the negation of (b).

Indeed, (a) is false if and only if there exists an ultrafilter 
$D$ over $S_\mu(\lambda )$ which covers 
 $\lambda $ and such that, for every $\beta \in \kappa $, 
$f_\beta(D)$ does not cover $ \lambda' $.

This means that $D$ covers $ \lambda $ and, for every $ \beta \in \kappa $, there exists some
$g(\beta) \in \lambda' $ such that 
$[\{g(\beta)\}] =\{ s' \in S _{\mu'}(\lambda' )| g( \beta )\in s'\} \not\in f_\beta(D) $,
that is, 
$ f_\beta ^{-1} [\{g(\beta)\}] \not\in D$,
that is, 
$ f_\beta ^{-1} \overline{[\{g(\beta)\}]} \in D$,
since $D$ is required to be an ultrafilter. Here,
$\overline{[\{g(\beta)\}]}$
denotes the complement of 
$[\{g(\beta)\}]$
in $ S _{\mu'}(\lambda' )$.

Thus, there exists some ultrafilter $D$ which makes (a) false if and only if
there exists some function $g:\kappa\to \lambda' $
such that the set
$ \{ f_\beta ^{-1} \overline{[\{g(\beta)\}]}| \beta \in \kappa  \}
\cup
\{ [\{ \alpha \}]| \alpha \in \lambda \} 
$ 
has the finite intersection property.

Equivalently, there exists some function $g:\kappa\to \lambda' $
such that, for every finite $F \subseteq \kappa $ and every
finite $G \subseteq \lambda $,
$\bigcap_{\beta\in F}  f_\beta ^{-1} \overline{[\{g(\beta)\}]}
\cap
\bigcap_{\alpha\in G} [\{ \alpha \}] \not= \emptyset $.

Since $\bigcap_{\alpha\in G} [\{ \alpha \}]=[G]$,
the negation of the above statement is: for every function $g: \kappa \to \lambda' $ there exist finite
sets $F \subseteq \kappa $ and   $G \subseteq \lambda $ such that 
$ [G] \cap \bigcap_{\beta\in F}  f_\beta ^{-1} \overline{[\{g(\beta)\}]}
= \emptyset $.

This is clearly equivalent to (b).
\end{proof}

\begin{proof}[Proof of Theorem \ref{lmklmeq}]
(a) $\Leftrightarrow$ (b) is immediate from Lemma \ref{lem}.

(b) $\Rightarrow$ (c)
For $ \alpha \in \lambda' $ and $\beta \in \kappa $,  define 
$C_{ \alpha , \beta } = f_\beta ^{-1} [\{\alpha \} ]$.

(c) $\Rightarrow$ (b) For $\beta\in\kappa $ define $f_\beta : S_\mu(\lambda ) \to S_{\mu'}(\lambda ')$ by 
$f_\beta (x)= \{  \alpha \in \lambda' | x \in C_{ \alpha  , \beta  }\}. $

(c) $\Leftrightarrow$ (c$'$) is immediate by taking complements in $ S_\mu(\lambda )$.

(e) $\Rightarrow$ (d) is trivial.

(d) $\Rightarrow$ (a).
Let ${\mathfrak S}^+ $
be an expansion of 
${\mathfrak S}(\lambda, \mu; \lambda',\mu') $
witnessing (d).

Without loss of generality, we can assume that 
${\mathfrak S}^+$ has Skolem functions (see \cite[Section 3.3]{CK}).
Indeed, since $ \kappa \geq \sup \{  \lambda , \lambda' \}  $, adding Skolem functions
to ${\mathfrak S}^+$ involves adding at most $ \kappa $ new symbols.

Consider the set of all functions $f: U \to U'$ 
which are definable in ${\mathfrak S}^+$. Enumerate them as 
$ (f_ \beta ) _{ \beta \in \kappa } $. We are going to show that 
these functions witness (a). 

Indeed, let $D$ be an ultrafilter
 over $ S_\mu(\lambda) $ which covers $ \lambda $. 
Consider the $D$-class $Id_D$ of the identity
function  on $U= S_\mu(\lambda) $. 
Since $D$ covers $ \lambda $, then in the model
  $ {\mathfrak C} = \prod_D {\mathfrak S}^+$  we have that
$ d( \{ \alpha \} ) \subseteq  Id_D$ for every $ \alpha \in  \lambda $, 
where $d$  denotes the elementary embedding.
Trivially, $ {\mathfrak C} \models U(Id_D)$.
 Let $ {\mathfrak B} $ be the Skolem hull
of $Id_D$ in ${\mathfrak C}$. 
Since 
${\mathfrak S}^+ $ has Skolem functions, 
$ {\mathfrak B} \preceq {\mathfrak C}$ \cite[Proposition 3.3.2]{CK};
in particular, $ {\mathfrak B} \equiv {\mathfrak C}$.
By 
\L o\v s Theorem,
$ {\mathfrak C} \equiv {\mathfrak S}^+ $.
 By transitivity,
$ {\mathfrak B} \equiv {\mathfrak S}^+ $.

Since $Id_D\in B$, then
$ \mathfrak B$ is $(\lambda,\mu)$-regular, by what we have proved.
Since ${\mathfrak S}^+$
witnesses
(d), then   
$ \mathfrak B$ has an element $x'_D$ such that
 $ {\mathfrak B} \models U'(x'_D)$ and
$ {\mathfrak B} \models \{\alpha'\} \subseteq  x'_D $ for every $ \alpha'  \in \lambda' $.

Since $ {\mathfrak B} $ is the Skolem hull
of $Id_D$ in ${\mathfrak C}$, we have $x'_D =f(Id_D)$,
 for
some function $f: S \to S$ definable in ${\mathfrak S}^+$,
where $S$ is the base set of  ${\mathfrak S}^+$.

Since $f$ is definable in ${\mathfrak S}^+$,
 then also the following function $f':U \to U'$
is definable in ${\mathfrak S}^+$:
\[
f'( u)=
\begin{cases}
f(u) & \textrm{ if } u \in U \textrm{ and } f(u ) \in U' \\
\emptyset & \textrm{ if } u \in U \textrm{ and }  f(u ) \not\in U'\\
\end{cases}
\]
Since $ {\mathfrak B} \models U'(x'_D)$, and
$ {\mathfrak B} \preceq {\mathfrak C}$, then
$ \{ u \in U = S_\mu ( \lambda )| x'(u) \in U'\} \in D$.
Since $x'_D =f(Id_D)$, then 
$ \{ u \in U | x'(u)=f(Id(u))= f(u)\} \in D$.
Hence, 
$ \{ u \in U | f(u ) \in U'\} \in D$,
$ \{ u \in U | f(u )=f'(u)\} \in D$ and
$ \{ u \in U | x'(u)=f'(u )\} \in D$.
Hence, $x'_D=f'_D$.

Since $f':U \to U'$ and
$f'$ is definable  in 
${\mathfrak S}^+$, then 
$f'= f_ \beta $ for some $ \beta \in \kappa $.
We want to show that $D'= f'(D)$ covers $\lambda '$.

Indeed, for $ \alpha' \in \lambda' $,
$[\{ \alpha '\}] = \{ u' \in S _{\mu'}(\lambda' )| \alpha'\in u'\} \in f'(D) $
if and only if 
 $\{ u \in S _{\mu}(\lambda )| \alpha'\in f'(u)\} \in D $, and this is true for every 
$ \alpha' \in \lambda' $, since 
$ {\mathfrak B} \models \{\alpha'\} \subseteq  x'_D $
and $x'_D=f'_D$.

(a) $\Rightarrow$ (e). Suppose we have functions
$ (f_ \beta ) _{ \beta \in \kappa } $ as given by \ref{lmklmeq}(a).

Expand 
${\mathfrak S}(\lambda, \mu; \lambda',\mu') $
to a model 
$ {\mathfrak S}^+$ 
by adding, for each $ \beta \in \kappa $, a new function symbol representing
$f_ \beta $ (by abuse of notation, in what follows we shall
write $f_ \beta $ both for the function itself and
for the symbol that represents it).

Suppose that $\mathfrak B \equiv {\mathfrak S}^+$ and 
$ \mathfrak B$ is $( \lambda, \mu)$-regular, that is, 
$ \mathfrak B$ has an element $x$ such that 
$ {\mathfrak B} \models U(x) $, and
$ {\mathfrak B} \models  \{ \alpha \} \subseteq   x $
for every $ \alpha \in \lambda $.

For every formula $ \phi (z)$ with just one variable $z$ 
in the language of  
$ {\mathfrak S}^+$ let 
$E_ \phi = \{ s \in S_{\mu}(\lambda ) | 
{\mathfrak S}^+ \models \phi ( s )\}  $.
Let $F= \{ E_ \phi | \mathfrak B \models \phi (x)\} $.
Since the intersection of any two members of $F$ is still in 
$F$, and $\emptyset \not\in F$, then $F$ can be extended to 
an ultrafilter $D$ on $ S_{\mu}(\lambda )$. 

For every $ \alpha\in  \lambda $,
consider the formula $ \phi (z) \equiv \{ \alpha \} \subseteq  z $.
 We get
$E_ \phi = \{ s \in S_{\mu}(\lambda ) |  
{\mathfrak S}^+ \models \{ \alpha \} \subseteq   s \}=
[\{\alpha \} ]  $. On the other side, 
since $ {\mathfrak B} \models  \{ \alpha \} \subseteq   x $, then
by the definition of $F$ we have
$ E_ \phi = [\{\alpha \} ] \in F \subseteq D$.
Thus, $D$ covers $ \lambda $.  

By (a), 
$ f_ \beta (D) $ covers  $ \lambda' $, for
some $ \beta \in \kappa$. This means that
$[\{\alpha' \} ] = \{ s' \in S_{\mu'}(\lambda' ) |\{\alpha' \} \subseteq s' \} \in f_ \beta (D)$, 
for every $ \alpha' \in \lambda'  $.
That is,
$\{ s \in S_{\mu}(\lambda ) | \{ \alpha' \} \subseteq   f_ \beta (s) \} \in D$ 
for every $ \alpha' \in \lambda'  $.

For every $ \alpha' \in \lambda'  $
 consider the formula 
$ \psi (z) \equiv  \{ \alpha' \} \subseteq   f_ \beta (z) $.
By the previous paragraph, 
$E_ \psi \in D$.
Notice that 
$E _{\neg \psi}  $ is the complement of 
$E_ \psi $ in $ S_ \mu (\lambda )$.
Since $D$ is proper,
and $ E_ \psi \in D$, then
$E _{\neg \psi} \not\in D $.
Since $D$ extends $F$,
and 
either $E_ \psi \in F$
or $E _{\neg \psi}  \in F$,
 we necessarily have
 $E_ \psi \in F$, that is, 
$\mathfrak B \models \psi (x)$, that is, 
$\mathfrak B \models  \{ \alpha' \} \subseteq   f_ \beta (x) $.

Moreover, since $ f_ \beta : S_\mu(\lambda ) \to S_{\mu'}(\lambda ')$,
$\mathfrak B \models U(x)$
and $ \mathfrak B \equiv \mathfrak S^+$, then 
$\mathfrak B \models U'(f_ \beta (x)) $.

Thus, we have proved that 
$ \mathfrak B$ has an element $y= f_ \beta (x)$ such that 
$\mathfrak B \models U'(y) $ and 
$ {\mathfrak B} \models  \{ \alpha' \} \subseteq y $ for every $ \alpha' \in \lambda' $.
This means that $ \mathfrak B$ is $( \lambda', \mu')$-regular.

The equivalence of (f) and (g) with the other conditions shall be proved elsewhere.
\end{proof}

\section{A multicardinal generalization}\label{2} 

Let us recall some generalizations of Definitions
\ref{lmklm} and \ref{deffromarx}, generalizations 
introduced in the statement of
\cite[Theorem 0.20]{arxiv}. 

\begin{definitions}\label{multilmklm}  
Recall from Definition \ref{lmklm} that  an ultrafilter $D$ over $S_\mu(\lambda )$
covers $\lambda $ if and only if $[\{ \alpha \}] \in D $,
for every $ \alpha \in \lambda $.

Suppose that $\kappa $ is any cardinal, and 
$\lambda$, $\mu$, $(\lambda_\beta )_{\beta \in \kappa}$, $(\mu_\beta )_{\beta \in \kappa}$ 
are infinite cardinals such that   $\lambda\geq\mu$,
 and  $\lambda_\beta \geq\mu_\beta $ for every $\beta \in \kappa $.

The relation 
$(\lambda, \mu) {\Rightarrow} \bigvee_{\beta \in \kappa } (\lambda_\beta , \mu_\beta )$ 
holds if and only if  there are $\kappa $ functions 
$(f_\beta )_{\beta \in \kappa }$
such that, for every $\beta\in\kappa, $
$f_\beta  : S_\mu(\lambda ) \to S_{\mu_\beta }(\lambda_\beta )$ and 
such that whenever $D$ is an ultrafilter
over $S_\mu(\lambda )$
which covers $\lambda $,
then for some $\beta\in\kappa $ it happens that
$f_\beta (D)$ covers $ \lambda _\beta $.

Notice that in the case when 
$\lambda_\beta =\lambda'$ and
$\mu_\beta = \mu'$
for every $\beta \in \kappa $
then 
$(\lambda, \mu) {\Rightarrow} \bigvee_{\beta \in \kappa } (\lambda_\beta , \mu_\beta )$ 
is the same as 
$(\lambda, \mu) \stackrel{\kappa}{\Rightarrow} (\lambda', \mu')$.

${\mathfrak S}(\lambda, \mu; \lambda_\beta ,\mu_\beta)_{\beta \in \kappa}$ 
 denotes the model
$\langle S_{\mu''}(\lambda ''), \subseteq, U, U_\beta , \{\alpha \} \rangle_{\alpha\in \lambda'',\beta \in \kappa}$,
where
$\lambda''= \sup  \{ \lambda, \sup_{\beta \in\kappa }\lambda_\beta  \} $,
$\mu''= \sup  \{ \mu, \sup_{\beta \in\kappa }\mu_\beta  \}$,
$U(x)$ if and only if $x \in S _\lambda (\mu)$, and, for $\beta\in\kappa $,
$U_\beta (y)$ if and only if $y\in  S _{ \lambda_\beta  } (\mu_\beta 
 )$. 

If ${\mathfrak S}^+ $ is an expansion of 
${\mathfrak S}(\lambda, \mu; \lambda_\beta ,\mu_\beta)_{\beta \in \kappa}$, 
and $ {\mathfrak B} \equiv {\mathfrak S}^+ $, then, for $ \beta \in \kappa $, 
we say that ${\mathfrak B}$ is $(\lambda_ \beta , \mu_ \beta )$-\emph{regular}  
if and only if there is $b \in B$ such that 
$ {\mathfrak B} \models U_ \beta (b) $, and
$ {\mathfrak B} \models  \{ \alpha \} \subseteq   b $
for every $ \alpha \in \lambda_ \beta  $.
The definition of a  $(\lambda, \mu)$-regular
extension is as in 
Definition 
\ref{deffromarx}
\end{definitions}

\begin{thm}\label{multilmklmeq} 
Suppose that $\kappa $ is any cardinal, and
$\lambda$, $\mu$, $(\lambda_\beta )_{\beta \in \kappa}$, $(\mu_\beta )_{\beta \in \kappa}$ 
are infinite cardinals such that   $\lambda\geq\mu$,
 and  $\lambda_\beta \geq\mu_\beta $ for every $\beta \in \kappa $.

Then the following conditions are equivalent.

\smallskip

(a) $(\lambda, \mu) {\Rightarrow} \bigvee_{\beta \in \kappa } (\lambda_\beta , \mu_\beta )$ holds.

\smallskip

(b) There are $\kappa $ functions 
$(f_\beta )_{\beta \in \kappa }$
such that, for every $\beta\in\kappa $,
$f_\beta  : S_\mu(\lambda ) \to S_{\mu_\beta }(\lambda_\beta )$ and 
such that 
for every function $g \in \prod _{\beta \in \kappa }\lambda _\beta   $ there exist finite
sets $F \subseteq \kappa $ and   $G \subseteq \lambda $ such that 
$ [G] \subseteq
 \bigcup_{\beta\in F}  f_\beta ^{-1} [\{g(\beta)\}]$.

\smallskip

(c) There is a family $ (C_{ \alpha  , \beta }) _{ \alpha\in\lambda_\beta   , \beta \in \kappa}  $ 
of subsets of $ S_ \mu(\lambda)$ such that:

(i) For every $ \beta\in\kappa$ and  every  $H \subseteq \lambda_ \beta  $, if
 $|H|\geq \mu_\beta $ then
 $\bigcap_{\alpha  \in H} C_{ \alpha  , \beta  } = \emptyset  $.

 (ii) For every function $g \in \prod _{\beta \in \kappa }\lambda _\beta   $ there exist finite
sets $F \subseteq \kappa $ and   $G \subseteq \lambda $ such that 
$ [G] \subseteq  \bigcup_{\beta\in F} C_{ g(\beta) , \beta  }$.

 \smallskip

(c$'$) There is a family $ (B_{ \alpha  , \beta }) _{ \alpha\in\lambda_\beta   , \beta \in \kappa}  $ 
of subsets of $ S_ \mu(\lambda)$ such that:

 (i) For every $ \beta\in\kappa$ and  every  $H \subseteq \lambda_\beta  $, if
 $|H|\geq \mu_\beta $ then
 $\bigcup_{\alpha  \in H} B_{ \alpha  , \beta  } = S_ \mu(\lambda) $.

 (ii) For every function $g \in \prod _{\beta \in \kappa }\lambda _\beta   $ there exist finite
sets $F \subseteq \kappa $ and   $G \subseteq \lambda $ such that 
$ [G] \cap \bigcap_{\beta\in F} B_{ g(\beta) , \beta  } 
= \emptyset $.

\smallskip

Suppose in addition that
$\kappa \geq \lambda $ and $\kappa \geq \lambda_\beta $
for all $\beta\in\kappa $,
and that, for every $\beta_0\in \kappa $,
$ | \{\beta\in\kappa| \lambda_\beta =\lambda_{\beta_0} \text{ and } 
\mu_\beta =\mu_{\beta_0} \}  |= \kappa $.
Then the preceding conditions are also equivalent to:

\smallskip

(d) ${\mathfrak S}(\lambda, \mu; \lambda_\beta ,\mu_\beta)_{\beta \in \kappa}$ 
 has 
an expansion
(equivalently, a multi-sorted expansion)
${\mathfrak S}^+ $ 
with at most $\kappa $ new symbols and such that whenever 
$ {\mathfrak B} \equiv {\mathfrak S}^+ $ and 
${\mathfrak B}$ is $(\lambda, \mu)$-regular, then 
there is $\beta \in \kappa $ such that 
${\mathfrak B}$ is $(\lambda_\beta , \mu_\beta )$-regular.

\smallskip

Suppose further that 
 $\lambda_\beta =\mu_\beta $ is a regular cardinal for every $\beta\in \kappa   $.
Then the preceding conditions are also equivalent to:

\smallskip

(e) Every $\kappa$-$(\lambda,\mu)$-compact logic is
 $\kappa$-$(\lambda_\beta ,\lambda_\beta )$-compact for some $\beta \in \kappa $. 

\smallskip

(f) Every $\kappa$-$(\lambda,\mu)$-compact logic generated by 
$\sup_{\beta\in\kappa }\lambda_\beta  $
cardinality quantifiers is 
$(\lambda_\beta ,\lambda_\beta )$-compact for some $\beta \in \kappa $. 
\end{thm} 

\begin{proof} There is no essential difference with the proof of 
Theorem \ref{lmklmeq}.
\end{proof} 

Notice that Theorem \ref{lmklmeq}.
is the particular case of Theorem \ref{multilmklmeq}
when 
$\lambda_\beta =\lambda'$ and
$\mu_\beta = \mu'$
for every $\beta \in \kappa $.

\section{The ``almost'' generalizations}\label{3}

Versions of Theorem \ref{lmklmeq} can be given for the ``almost'' variants of 
$(\lambda, \mu) \stackrel{\kappa}{\Rightarrow} (\lambda', \mu')$. In order to state the above remark precisely, we need to introduce some variations on 
Definitions \ref{lmklm} and \ref{deffromarx} (see \cite[Definition 0.14]{arxiv}).

\begin{definition}\label{almdef}    
We say that an ultrafilter $D$ over $S _{\mu}(\lambda )$
\emph{almost covers} $\lambda $ if and only if 
$| \{ \alpha \in \lambda| [\alpha ] \in D \}|= \lambda  $.

The relation $\alm (\lambda, \mu) \stackrel{\kappa}{\Rightarrow} (\lambda', \mu')$ holds if and only if 
there are $\kappa $ functions 
$(f_\beta )_{\beta \in \kappa }: S_\mu(\lambda ) \to S_{\mu'}(\lambda ')$
such that whenever $D$ is an ultrafilter over $S_\mu(\lambda )$,
 and $D$ almost covers $\lambda $,
then for some $\beta\in\kappa $ it happens that
$f_\beta (D)$ covers $\lambda '$.

The relations $\alm (\lambda, \mu) \stackrel{\kappa}{\Rightarrow} \alm (\lambda', \mu')$ and $(\lambda, \mu) \stackrel{\kappa}{\Rightarrow} \alm (\lambda', \mu')$
are defined similarly.
 
Notice that $(\lambda, \mu) \stackrel{1}{\Rightarrow} \alm (\lambda, \mu)$ trivially.
Moreover, 
if $\nu $ is a regular cardinal, then
$\alm (\nu, \nu ) \stackrel{1}{\Rightarrow} (\nu, \nu )$,
as witnessed by $f: S_\nu(\nu) \to S_\nu(\nu)$ 
defined by $f(x)=\sup x$
(cf. also \cite[Lemma 0.16(ii)]{arxiv}).
Thus, in the next results, if either $\nu=\lambda=\mu $, 
or $\nu'=\lambda'=\mu' $, or both, regular cardinals, then 
$\alm (\nu, \nu ) $ and $ (\nu, \nu )$ can be used interchangeably.
A similar remark applies to almost $(\nu, \nu)$-regularity and $(\nu, \nu)$-regularity, as defined below.
 
Recall the definition of the model 
${\mathfrak S}(\lambda, \mu; \lambda',\mu') $
from Definition \ref{deffromarx}.

If ${\mathfrak S}^+ $ is an expansion of 
${\mathfrak S}(\lambda, \mu; \lambda'\mu') $,
and $ {\mathfrak B} \equiv {\mathfrak S}^+ $,
we say that  ${\mathfrak B}$ is \emph{almost} $(\lambda, \mu)$-\emph{regular}  
if and only if there is $b \in B$ such that 
$ {\mathfrak B} \models U(b) $, and
$ |\{\alpha \in \lambda |
 {\mathfrak B} \models  \{ \alpha \} \subseteq   b \}| =  \lambda $.
The notion of 
\emph{almost} $(\lambda', \mu')$-\emph{regularity} 
is defined similarly.
\end{definition}

\begin{thm}\label{lmkalmlmeq} 
Suppose that  $ \lambda \geq \mu$, $ \lambda' \geq \mu'$
 are infinite cardinals, and $ \kappa $ is any cardinal. 

Then the following conditions are equivalent.

\smallskip

(a) $(\lambda, \mu) \stackrel{\kappa}{\Rightarrow} \alm (\lambda', \mu')$ holds.

\smallskip

(b) There are $ \kappa $ functions 
$ (f_ \beta ) _{ \beta \in \kappa }: S_\mu(\lambda ) \to S_{\mu'}(\lambda ')$
such that 
for every function $g: \kappa \to S _{ \lambda' } (\lambda') $ there exist a finite
set  $G \subseteq \lambda $, a finite set $F \subseteq \kappa $, and, for
$ \beta \in F$, 
finite sets $H_ \beta \subseteq \lambda'\setminus g( \beta ) $ 
  such that 
$ [G] \subseteq
 \bigcup_{\beta\in F, \beta ^* \in H_ \beta }  f_\beta ^{-1} [\{g(\beta^*)\}]$.

\smallskip

(c) There is a family $ (C_{ \alpha  , \beta }) _{ \alpha\in\lambda'  , \beta \in \kappa}  $ 
of subsets of $ S_ \mu(\lambda)$ such that:

(i) For every $ \beta\in\kappa$ and  every  $H \subseteq \lambda' $, if
 $|H|\geq \mu'$ then
 $\bigcap_{\alpha  \in H} C_{ \alpha  , \beta  } = \emptyset  $.

 (ii) For every function $g: \kappa \to S _{ \lambda' } (\lambda') $ there exist a finite
set  $G \subseteq \lambda $, a finite set $F \subseteq \kappa $, and, for
$ \beta \in F$, 
finite sets $H_ \beta \subseteq \lambda'\setminus g( \beta ) $ 
  such that 
$ [G] \subseteq  \bigcup_{\beta\in F,\beta ^* \in H_ \beta } C_{ g(\beta^*) , \beta }$.

 \smallskip

(c$'$) 
There is a family $ (B_{ \alpha  , \beta }) _{ \alpha\in\lambda'  , \beta \in \kappa}  $ 
of subsets of $ S_ \mu(\lambda)$ such that:

(i) For every $ \beta\in\kappa$ and  every  $H \subseteq \lambda' $, if
 $|H|\geq \mu'$ then
 $\bigcup_{\alpha  \in H} B_{ \alpha  , \beta  } = S_ \mu(\lambda) $.

 (ii) For every function $g: \kappa \to S _{ \lambda' } (\lambda') $ there exist a finite
set  $G \subseteq \lambda $, a finite set $F \subseteq \kappa $, and, for
$ \beta \in F$, 
finite sets $H_ \beta \subseteq \lambda'\setminus g( \beta ) $ 
  such that 
$ [G] \cap \bigcap_{\beta\in F, \beta ^* \in H_ \beta } B_{ g(\beta^*) , \beta  } 
= \emptyset $.

\smallskip

If in addition
$\kappa \geq \sup \{  \lambda , \lambda' \} $,
then the preceding conditions are also equivalent to:

\smallskip

(d) ${\mathfrak S}(\lambda, \mu; \lambda',\mu') $ has 
an expansion
(equivalently, a multi-sorted expansion)
${\mathfrak S}^+ $ 
with at most $\kappa $ new symbols such that
whenever 
$ {\mathfrak B} \equiv {\mathfrak S}^+ $ and ${\mathfrak B}$ is $(\lambda, \mu)$-regular, then 
${\mathfrak B}$ is almost $(\lambda', \mu')$-regular.
\end{thm} 

\begin{thm}\label{almlmklmeq} 
Suppose that  $ \lambda \geq \mu$, $ \lambda' \geq \mu'$
 are infinite cardinals, and $ \kappa $ is any cardinal. 

Then the following conditions are equivalent.

\smallskip

(a) $\alm (\lambda, \mu) \stackrel{\kappa}{\Rightarrow} (\lambda', \mu')$ holds.

\smallskip

(b) There are $ \kappa $ functions 
$ (f_ \beta ) _{ \beta \in \kappa }: S_\mu(\lambda ) \to S_{\mu'}(\lambda ')$
such that 
for every $T \subseteq \lambda $ with $|T|= \lambda $, and  
for every function $g: \kappa \to \lambda' $ there exist finite
sets $F \subseteq \kappa $ and   $G \subseteq T$ such that 
$ [G] \subseteq
 \bigcup_{\beta\in F}  f_\beta ^{-1} [\{g(\beta)\}]$.

\smallskip

(c) There is a family $ (C_{ \alpha  , \beta }) _{ \alpha\in\lambda'  , \beta \in \kappa}  $ 
of subsets of $ S_ \mu(\lambda)$ such that:

(i) For every $ \beta\in\kappa$ and  every  $H \subseteq \lambda' $, if
 $|H|\geq \mu'$ then
 $\bigcap_{\alpha  \in H} C_{ \alpha  , \beta  } = \emptyset  $.

 (ii) For every function $g: \kappa \to \lambda' $ 
and for every $T \subseteq \lambda $ with $|T|= \lambda $
there exist finite
sets $F \subseteq \kappa $ and   $G \subseteq T$ such that 
$ [G] \subseteq  \bigcup_{\beta\in F} C_{ g(\beta) , \beta  }$.

 \smallskip

(c$'$) There is a family $ (B_{ \alpha  , \beta }) _{ \alpha\in\lambda'  , \beta \in \kappa}  $ 
of subsets of $ S_ \mu(\lambda)$ such that:

(i) For every $ \beta\in\kappa$ and  every  $H \subseteq \lambda' $, if
 $|H|\geq \mu'$ then
 $\bigcup_{\alpha  \in H} B_{ \alpha  , \beta  } = S_ \mu(\lambda) $.

 (ii) For every function $g: \kappa \to \lambda' $ and
for every $T \subseteq \lambda $ with $|T|= \lambda $
there exist finite
sets $F \subseteq \kappa $ and   $G \subseteq T $ such that 
$ [G] \cap \bigcap_{\beta\in F} B_{ g(\beta) , \beta  } 
= \emptyset $.

\smallskip

If in addition
$\kappa \geq \sup \{  \lambda , \lambda' \} $,
then the preceding conditions are also equivalent to:

\smallskip

(d) ${\mathfrak S}(\lambda, \mu; \lambda',\mu') $ has 
an expansion (equivalently, a multi-sorted expansion)
${\mathfrak S}^+ $ 
with at most $\kappa $ new symbols such that whenever 
$ {\mathfrak B} \equiv {\mathfrak S}^+ $ and 
${\mathfrak B}$ is almost $ (\lambda, \mu)$-regular, then 
${\mathfrak B}$ is $(\lambda', \mu')$-regular.

\smallskip

If in addition
$\kappa \geq \sup \{  \lambda , \lambda' \} $, and $\lambda'=\mu'$ is a regular
cardinal
then the preceding conditions are also equivalent to:

\smallskip

(f) Every almost $\kappa$-$(\lambda,\mu)$-compact logic is  
$\kappa$-$(\lambda',\lambda')$-compact. 

\smallskip

(g) Every almost $\kappa$-$(\lambda,\mu)$-compact logic generated by $\lambda' $
cardinality quantifiers is $(\lambda',\lambda')$-compact. 
\end{thm} 

\begin{thm}\label{almlmkalmlmeq} 
Suppose that  $ \lambda \geq \mu$, $ \lambda' \geq \mu'$
 are infinite cardinals, and $ \kappa $ is any cardinal. 

Then the following conditions are equivalent.

\smallskip

(a) $ \alm (\lambda, \mu) \stackrel{\kappa}{\Rightarrow} \alm (\lambda', \mu')$ holds.

\smallskip

(b) There are $ \kappa $ functions 
$ (f_ \beta ) _{ \beta \in \kappa }: S_\mu(\lambda ) \to S_{\mu'}(\lambda ')$
such that 
for every $T \subseteq \lambda $ with $|T|= \lambda $ and
for every function $g: \kappa \to S _{ \lambda' } (\lambda') $ there exist a finite
set  $G \subseteq T$, a finite set $F \subseteq \kappa $, and, for
$ \beta \in F$, 
finite sets $H_ \beta \subseteq \lambda'\setminus g( \beta ) $ 
  such that 
$ [G] \subseteq
 \bigcup_{\beta\in F, \beta ^* \in H_ \beta }  f_\beta ^{-1} [\{g(\beta^*)\}]$.

\smallskip

(c) There is a family $ (C_{ \alpha  , \beta }) _{ \alpha\in\lambda'  , \beta \in \kappa}  $ 
of subsets of $ S_ \mu(\lambda)$ such that:

(i) For every $ \beta\in\kappa$ and  every  $H \subseteq \lambda' $, if
 $|H|\geq \mu'$ then
 $\bigcap_{\alpha  \in H} C_{ \alpha  , \beta  } = \emptyset  $.

 (ii) For every $T \subseteq \lambda $ with $|T|= \lambda $ and
for every function $g: \kappa \to S _{ \lambda' } (\lambda') $ there exist a finite
set  $G \subseteq T $, a finite set $F \subseteq \kappa $, and, for
$ \beta \in F$, 
finite sets $H_ \beta \subseteq \lambda'\setminus  g( \beta ) $ 
  such that 
$ [G] \subseteq  \bigcup_{\beta\in F,\beta ^* \in H_ \beta } C_{ g(\beta^*) , \beta }$.

 \smallskip

(c$'$) 
There is a family $ (B_{ \alpha  , \beta }) _{ \alpha\in\lambda'  , \beta \in \kappa}  $ 
of subsets of $ S_ \mu(\lambda)$ such that:

(i) For every $ \beta\in\kappa$ and  every  $H \subseteq \lambda' $, if
 $|H|\geq \mu'$ then
 $\bigcup_{\alpha  \in H} B_{ \alpha  , \beta  } = S_ \mu(\lambda) $.

 (ii) For every $T \subseteq \lambda $ with $|T|= \lambda $ and
for every function $g: \kappa \to S _{ \lambda' } (\lambda') $ there exist a finite
set  $G \subseteq T $, a finite set $F \subseteq \kappa $, and, for
$ \beta \in F$, 
finite sets $H_ \beta \subseteq \lambda'\setminus g( \beta ) $ 
  such that 
$ [G] \cap \bigcap_{\beta\in F, \beta ^* \in H_ \beta } B_{ g(\beta^*) , \beta  } 
= \emptyset $.

\smallskip

If in addition
$\kappa \geq \sup \{  \lambda , \lambda' \} $,
then the preceding conditions are also equivalent to:

\smallskip

(d) ${\mathfrak S}(\lambda, \mu; \lambda',\mu') $ has 
an expansion
(equivalently, a multi-sorted expansion)
${\mathfrak S}^+ $ 
with at most $\kappa $ new symbols such that
whenever 
$ {\mathfrak B} \equiv {\mathfrak S}^+ $ and ${\mathfrak B}$ is almost 
$(\lambda, \mu)$-regular, then 
${\mathfrak B}$ is almost $(\lambda', \mu')$-regular.
\end{thm}

\begin{proof}[Proofs] Similar to the proof of 
Theorem \ref{lmklmeq}.
\end{proof} 

\begin{remark}\label{common}
There is a common generalization of Theorems \ref{multilmklmeq}
and \ref{lmkalmlmeq}. 
There is a simultaneous generalization of Theorems 
\ref{almlmklmeq} and \ref{almlmkalmlmeq} along the lines of Theorem 
\ref{multilmklmeq}.
We leave details to the reader.

In fact, we have a common generalization of all the results 
presented in this note, including Remark \ref{rmkult} below.
Details shall be presented elsewhere.
 \end{remark}

\section{Two problems and a further generalization}\label{prob}

\begin{problem}\label{prob1}
 It is proved in \cite[Theorem 2]{pams}
that 
$ \alm (\lambda^+, \mu^+) \stackrel{\lambda^+}{\Rightarrow} \alm (\lambda, \mu)$ holds.

Is it true that 
$(\lambda^+, \mu^+) \stackrel{\lambda^+}{\Rightarrow} (\lambda, \mu)$ holds?

This is true when $\lambda=\mu$ is a regular cardinal.
\end{problem}

\begin{problem}\label{prob2}
 As proved in \cite{fms}, it is consistent, modulo some large cardinal assumption,
that there is a uniform ultrafilter over $\omega_1$ which is not
$(\omega,\omega_1 )$-regular. This is equivalent to the failure of
$(\omega_1,\omega_1) \stackrel{2^{\omega_1}}{\Rightarrow} (\omega_1, \omega )$.

Is it possible to find a model for the failure of
$(\omega_1,\omega_1) \stackrel{\omega_1}{\Rightarrow} (\omega_1, \omega )$
by using weaker consistency assumptions?

Which is the exact consistency strength of the failure of
$(\omega_1,\omega_1) \stackrel{\omega_1}{\Rightarrow} (\omega_1, \omega )$?

More generally, for arbitrary $\lambda $, which is the exact consistency strength of the failure of 
$(\lambda^+, \lambda ^+) \stackrel{\lambda^+}{\Rightarrow} (\lambda^+, \lambda )$?
 \end{problem}

\begin{remark}\label{rmkult}
We can extend the definitions and the results of the present note as follows.
If $ \lambda$ and $\mu$ are \emph{ordinals}, let 
$S_\mu(\lambda )$ denote the set of all subsets of 
$ \lambda$ having order type $<\mu$. See also \cite{BK}.

Definitions 
\ref{lmklm},
\ref{deffromarx},
\ref{multilmklm} 
and \ref{almdef} 
can be easily generalized to the case when
$ \lambda$ and $\mu$ are ordinals.
In the present situation, the most appropriate definition of ``almost covering'' appears to be the following:
an ultrafilter $D$ over $S _{\mu}(\lambda )$
almost covers $\lambda $ if and only if 
order type of $ \{ \alpha \in \lambda| [\alpha ] \in D \}= \lambda  $.

The definition of 
$\kappa$-$(\lambda,\mu)$-compactness for logics, too,
can be easily generalized 
when
$ \lambda$ and $\mu$ are ordinals,
by always taking into account order type, instead
of cardinality.

Notice that, in all the definitions and the results here, $\kappa$
is used just as an index set; the cardinal structure on $\kappa$
is not used at all. Hence, allowing $\kappa$
to be an ordinal is no gain in generality.

All the results of the present note, when appropriately formulated, extend to
the more general setting
when
$ \lambda, \mu, \lambda', \mu', \lambda_\beta, \mu_\beta$ are ordinals.

Everywhere, cardinality assumptions must be replaced by assumptions about order type.
For example, the condition 
$|H|\geq \mu'$ in Theorem \ref{lmklmeq} (c), (c$'$)
has to be replaced by ``the order type of $H$ is $\geq \mu'$''.

The condition
$\kappa \geq \sup \{  \lambda , \lambda' \} $
before clause (d) in Theorem \ref{lmklmeq}
can be replaced by
$\kappa \geq \sup \{|  \lambda |,| \lambda'| \} $.

The same applies to the condition before
clause (f), and we actually need the requirement
that $\lambda'=\mu'$ is a regular \emph{cardinal}.

Similar remarks apply to  
Theorems   
 \ref{multilmklmeq} and
\ref{almlmklmeq}.

Theorems
\ref{lmkalmlmeq}
and \ref{almlmkalmlmeq}
hold, too, with slight further modifications.
\end{remark}


\begin{thebibliography}{} 

%\bibitem[AU]{AU} P. Alexandroff, P. Urysohn, {\em M\'emorie sur les \'espaces topologiques compacts}, Ver. Akad. Wetensch. Amsterdam \textbf{14} (1929), 1-96. 

\bibitem[BF]{BF} J. Barwise, S. Feferman (editors),
\emph{Model-theoretic logics.}
 Perspectives in Mathematical
 Logic. Springer-Verlag, New York, 1985.

\bibitem
[BK]
{BK} M. Benda and J. Ketonen,  Regularity of ultrafilters, Israel J. 
Math.
{\bf 17}, 231--240 (1974).

%\bibitem[C1]{Cprepr} X. Caicedo, {\em On productive $[\kappa,\lambda]$-compactness, or  the Abstract Compactness Theorem revisited},  manuscript (1995). 

%\bibitem[C2]{C} X. Caicedo, {\em The Abstract Compactness Theorem Revisited}, in {\em Logic and Foundations of Mathematics (A. Cantini et al. editors),} Kluwer Academic Publishers  (1999), 131--141. 

\bibitem
[CK]
 {CK}C. C. Chang and J. Keisler,  Model Theory, Amsterdam (1977).

\bibitem[CN]{CN}  W. Comfort, S. Negrepontis, {\em The Theory of Ultrafilters}, Berlin (1974). 

%\bibitem[E]{Eb} H.-D. Ebbinghaus, \emph{Extended logics:   the general framework}, Chapter II in \cite{BF}, pp. 25--76.

%\bibitem[EU]{EU} P. Erd\"os, S. Ulam, {\em On equations with sets as unknowns}, Proc. Nat. Acad. Sci. U.S.A.  \textbf{60} (1968), 1189--1195. 

%\bibitem[HNV]{EGT} K. P. Hart, J. Nagata, J. E. Vaughan (editors), {\em Encyclopedia of General Topology}, Amsterdam (2003). 

\bibitem[FMS]{fms}
 M. Foreman, M. Magidor and S. Shelah,  Martin's Maximum, saturated
ideals and non-regular ultrafilters. Part II, Annals of Mathematics 
{\bf 127},
521--545, (1988).

\bibitem[KM]{KM} A. Kanamori and M. Magidor,  The evolution of large cardinal axioms in
Set Theory, in:  Higher Set Theory, edited by  G. H. M\"uller and D. S. Scott,
99--275, Berlin (1978).

%\bibitem[KV] {KV} K. Kunen and J. E. Vaughan (editors), {\em  Handbook of Set Theoretical Topology}, Amsterdam (1984).

\bibitem[L1]{easter} P. Lipparini, {\em About some generalizations of ($\lambda$, $\mu$)-compactness}, Proceedings of the 5$^{th}$ Easter conference on model theory (Wendisch Rietz, 1985), Seminarber., Humboldt-Univ. Berlin, Sekt. Math. {\bf 93}, 139--141 (1987). Available also at the author's web page.

%\bibitem[L2]{jsl} P. Lipparini, {\em Limit ultrapowers and abstract   logics}, J. Symbolic Logic {\bf 52} (1987), no. 2, 437--454.

\bibitem[L2]{bumi} P. Lipparini, {\em The compactness spectrum of abstract logics, large 
cardinals and combinatorial principles}, Boll. Unione Matematica Italiana ser. 
VII, {\bf 4-B} 875--903 (1990).

\bibitem[L3]{arxiv} P. Lipparini,  
{\em Ultrafilter translations, I: $(\lambda,\lambda )$-compactness 
of logics with a cardinality quantifier}, Arch. Math. Logic {\bf 35}, 
63--87 (1996).

%\bibitem[L4] {topproc} P. Lipparini,  {\em Productive $[\lambda,\mu ]$-compactness and regular ultrafilters}, Topology Proceedings \textbf{21} (1996), 161--171.
 
\bibitem[L4] {abst} P. Lipparini,  {\em Regular ultrafilters and $[ \lambda , \lambda ]$-compact products of topological spaces} (abstract), Bull. Symbolic Logic \textbf{5} (1999), 121.

%\bibitem[L5]{topappl} P. Lipparini,  {\em Compact factors in finally compact products of topological spaces}, Topology and its Applications, \textbf{153} (2006), 1365--1382.

\bibitem[L5]{parti} P. Lipparini,  {\em Combinatorial and model-theoretical principles related to 
regularity of ultrafilters and compactness of topological spaces. I}, 
arXiv:0803.3498; {\em II.}:0804.1445; {\em III.}:0804.3737; {\em IV.}:0805.1548  
(2008).

\bibitem[L6]{pams} P. Lipparini,  {\em Every $(\lambda ^+,\kappa ^+)$-regular ultrafilter is
$(\lambda ,\kappa )$-regular},  Proc. Amer. Math. Soc. \textbf{128} (1999), 605--609.

%\bibitem[Ma]{Ma} J. A. Makowsky, \emph{Compactness, embeddings and definability}, Chapter XVIII in \cite{BF}, pp. 645--716.  

%\bibitem[Ste]{steph} R. M. Stephenson, {\em Initially $\kappa $-compact
% and related spaces}, Chapter 13 in \cite{KV}, 603--632. 

%\bibitem[V1]{VLNM} J. E. Vaughan, {\em Some recent results in the theory of  [a,b]-compactness},
%in {\em TOPO 72---General Topology and its Applications} 
%(Proc. Second Pittsburg Internat. Conf., Carnegie-Mellon Univ. and Univ. Pittsburg, 1972), Lecture Notes in Mathematics 
% \textbf{378} (1974), 534--550. 

%\bibitem[V2]{Vfund} J. E. Vaughan, {\em Some properties
%related to [a,b]-compactness}, Fund. Math. \textbf{87} (1975), 251--260. 

\end{thebibliography}
\end{document}